\documentclass[11pt]{amsart}
 \usepackage{amsopn}
\usepackage{amsfonts,epsfig}
\usepackage{amsmath,amssymb,amsthm, latexsym,graphicx,graphics}
\usepackage{multicol}
\usepackage{color}
\usepackage{url,hyperref}
\usepackage{euscript}
\def \R{\mathbb{R}}

\newtheorem{theorem}{Theorem}[section]
\newtheorem{cor}[theorem]{Corollary}

\newtheorem{lema}[theorem]{Lemma}
\newtheorem{rem}[theorem]{Remark}
\newtheorem{remark}[theorem]{Remark}

\newtheorem{prop}[theorem]{Proposition}
\newtheorem{thm}{Theorem}

\usepackage{url,hyperref}
\hypersetup{}
\usepackage{euscript,color}

\definecolor{red}{rgb}{1,0,0}

\hypersetup{
    bookmarks=false,         % show bookmarks bar?
    unicode=false,          % non-Latin characters in Acrobats bookmarks
    pdftoolbar=true,        % show Acrobats toolbar?
    pdfmenubar=true,        % show Acrobats menu?
    pdffitwindow=false,     % window fit to page when opened
    pdfstartview={FitH},    % fits the width of the page to the window
    pdftitle={My title},    % title
    pdfauthor={Author},     % author
    pdfsubject={Subject},   % subject of the document
    pdfcreator={Creator},   % creator of the document
    pdfproducer={Producer}, % producer of the document
    pdfkeywords={keyword1} {key2} {key3}, % list of keywords
    pdfnewwindow=true,      % links in new window
    colorlinks=true,       % false: boxed links; true: colored links
    linkcolor=black,          % color of internal links
    citecolor=green,        % color of links to bibliography
    filecolor=magenta,      % color of file links
    urlcolor=red         % color of external links
}

\def\R{\mathbb{R}}

\begin{document}

%-------------------------------------------------------
\title[Mok's characteristic varieties and the normal holonomy group]
{Mok's characteristic varieties and the normal holonomy group}

\author[A. J. Di Scala]{Antonio J. Di Scala}
\author[F. Vittone]{Francisco Vittone}


\date{\today}

\dedicatory{Dedicated to the memory of Professor Guy Roos}

\maketitle

\begin{abstract}
In this paper we complete the study of the normal holonomy groups of complex submanifolds (non nec. complete) of $\mathbb{C}^n$ or $\mathbb{C}\mathbb{P}^n$.  We show that irreducible but non transitive normal holonomies are exactly the Hermitian $s$-representations of \cite[Table 1]{CD} (see Corollary \ref{tabla}). For each one of them we construct a non necessarily complete complex submanifold whose normal holonomy is the prescribed s-representation.

We also show that if the submanifold has irreducible non transitive normal holonomy then it is an open subset of the smooth part of one of the characteristic varieties studied by N. Mok in his work about rigidity of locally symmetric spaces.

Finally, we prove that if the action of the normal holonomy group of a projective submanifold is reducible then the submanifold is an open subset of the smooth part of a so called join, i.e. the union of the lines joining two projective submanifolds.
\end{abstract}

\section{Introduction}

Given a submanifold $M$ of a Riemannian manifold $N$, the normal bundle $\nu M$ carries a natural connection $\nabla^{\perp}$ defined
as the projection of the Levi-Civita connection on the tangent bundle $\mathrm{T}N$ to the normal bundle $\nu M$.
The \textsl{normal holonomy group} $Hol(M,\nabla^{\bot})$ of $M$ is the holonomy group of the connection $\nabla^{\perp}$. Its connected component of the identity $Hol^*(M,\nabla^{\bot})$ is the \textsl{restricted normal holonomy group} of $M$.

In \cite{Ol90} C. Olmos proved the \emph{Normal Holonomy Theorem} for submanifolds of real space forms.
It asserts that the action of the restricted normal holonomy group on any normal space $\nu_p M$ is the holonomy representation of a
Riemannian symmetric space. This result plays an important role in the theory of isoparametric submanifolds (see \cite{BCO}).

A similar result  for complex submanifolds of complex space forms was proved in \cite{AD} and improved in \cite{DV}. Namely, if $M$ is a full complex submanifold of a complex space form, then its restricted normal holonomy group acts on each normal space as the isotropy representation of a Hermitian symmetric space without flat factor.

\medskip
Therefore a natural problem arises: to determine which among all isotropy representations of  Hermitian symmetric spaces is the normal holonomy of a complex submanifold of $\mathbb{C}^n$ or $\mathbb{C}\mathbb{P}^n$. That is to say, for each Hermitian isotropy representation to decide if there exists a complex submanifold with it as normal holonomy.

\medskip

If the normal holonomy group of a complex submanifold $M$ of $\mathbb{C}^n$ or $\mathbb{C}\mathbb{P}^n$ acts transitively, on the unit sphere of the normal space, then  it must be the whole unitary group $U(k)$, where $k$ is the codimension of $M$. Therefore it is interesting to study submanifolds whose normal holonomy is \textsl{non transitive}.

Concrete non trivial examples were studied in \cite{CD}, where the normal holonomy group of complex parallel submanifolds of the complex projective space were computed.  The classification and realization problems for complete irreducible complex submanifolds were completely solved in \cite{CDO}.

 In \cite[Theorem 2]{CDO} it was proved that a complete, irreducible and full complex submanifolds of $\mathbb{C}^n$ must have transitive normal holonomy.
For the case of complete complex submanifolds of the projective space, it was proved a Berger type theorem.
The main result asserts that the action of the normal holonomy group is non transitive if and only if the submanifold is the complex orbit of the isotropy representation of an irreducible Hermitian symmetric space of rank greater than 2.
 Notice that such complex orbits are actually the first characteristic varieties, studied by N. Mok \cite{MokLibro}, whose normal holonomies are those explicitly computed in \cite{CD}.

It is important to point out that both in the case of a submanifold of $\mathbb{C}^n$ or of $\mathbb{C}\mathbb{P}^n$, these results are false if the completeness of the submanifold is not assumed (cf. \cite[Section 5]{CDO}).

So it is natural to ask which non transitive normal holonomy representations can appear if the complex submanifold is not complete.
Or more generally, if the non transitivity of the normal holonomy of a non complete complex submanifold implies that the submanifold belongs to a short list of submanifolds as in the case of being complete.

\medskip

The answer to these questions are the main results of this paper: the normal holonomy of an irreducible complex submanifold of $\mathbb{C}^n$ or $\mathbb{C}\mathbb{P}^n$ is non transitive if and only if the submanifold is an open subset of one of the (cones over a) Mok's characteristic variety studied in \cite{MokLibro}. Namely:

\begin{thm}\label{Euclideo}
Let $M \subset \mathbb{C}^n$ be a full and irreducible complex submanifold (non necessarily complete w.r.t. the induced metric of $\mathbb{C}^n$). Let $Hol^*(M, \nabla^{\perp})$ be the restricted normal holonomy group of $M$.
If the action of $Hol^*(M, \nabla^{\perp})$  is non transitive on the unit sphere of the normal space then there exists an irreducible bounded symmetric domain $D \subset \mathbb{C}^{n}$ (realized as a circled domain) such that $M$
is an open subset of the smooth part of the  Mok's characteristic cone $\mathcal{C}S^j(D)$ for $1 \leq j < rank(D) - 1$.

Conversely, for any irreducible bounded symmetric domain $D\subset \mathbb{C}^n$,  the restricted normal holonomy group of an open subset of the smooth part of the cone $\mathcal{C}S^j(D)$ for $1 \leq j < rank(D) - 1$ acts irreducibly but non transitively on the unit sphere of each normal space.
\end{thm}

\begin{thm}{\label{Projectivo}} Let $M \subset \mathbb{C}\mathbb{P}^n$ be a full complex submanifold (non nec. complete) whose restricted normal holonomy group acts irreducibly on the normal space. If this action is non transitive on the unit sphere of the normal space, then there exists an irreducible bounded symmetric domain $D \subset \mathbb{C}^{n+1}$ (realized as circled domain) such that $M$
is an open subset of the smooth part of the Mok's characteristic variety $S^j(D)$ for $1 \leq j < rank(D) - 1$.

Conversely, the normal holonomy group of an open subset of the smooth part of the Mok's characteristic variety $S^j(D)$ for $1 \leq j < rank(D) - 1$ acts irreducibly but not transitively on the unit sphere of the normal space.
\end{thm}

 As it will become clear in the proof of the above theorems, the normal holonomy of a $j^{th}$ Mok's characteristic variety $S^j(D)$, over a bounded symmetric domain $D$, always coincide with the normal holonomy of a first Mok's characteristic variety $S^1(D')$ (over a suitable bounded symmetric domain $D'$ different from $D$). Therefore, combining the above theorems with the results in \cite{CD}, we get the following classification result:

\begin{cor}\label{tabla} Let $M$ be a full complex submanifold  (non nec. complete) of either $\mathbb{C}^n$ or $\mathbb{C}\mathbb{P}^n$ whose normal holonomy is irreducible.
Then the normal holonomy group representation is either the full unitary group of the normal bundle $\mathrm{U}(\nu_pM)$ or one of the following
isotropy representations $K \hookrightarrow SO(V)$ of
compact irreducible Hermitian symmetric spaces $G/K$:

\medskip

\begin{center}

\begin{tabular}{|c|c|c|}
\hline
&&\\[-3.7mm]
\hspace*{10mm} $G/K$ \hspace*{10mm} & \hspace*{10mm} $K$
\hspace*{10mm} & \hspace*{15mm} $ V \hspace*{15mm}$
\\[0.3mm]
\hline
&& \\[-3.7mm]
$SU(p+q)/S(U(p)\times U(q))$, $p,q > 1$ & $S(U(p)\times
U(q))$
& $ \mathbb{C}^p \otimes \mathbb{C}^q$ \\[0.3mm]
\hline
&& \\[-3.7mm]
$SO(2n)/U(n)$, $n>3$ & $U(n)$ & $\Lambda^2(\mathbb{C}^n) $ \\[0.3mm]
\hline
&& \\[-3.7mm]
$SO(12)/SO(2)\times SO(10)$ & $SO(2)\times SO(10)$
&  $\mathbb{R}^2 \otimes \mathbb{R}^{10} $  \\[0.3mm]
\hline
&& \\[-3.7mm]
$Sp(n)/U(n)$, $n>1$ & $U(n)$
&  $S^2\mathbb{C}^n $  \\[0.3mm]
\hline

\end{tabular}
\end{center}
\medskip

 Moreover, any of this isotropy representations can be realized as the normal holonomy of a complex submanifold
as explained in Theorems \ref{Euclideo} and \ref{Projectivo}.

\end{cor}

\medskip

\medskip

It is interesting to notice that the isotropy representation of a quadric $$ Q_n := SO(n+2)/SO(2) \times SO(n) $$
can be realized as the normal holonomy of a complex submanifold only when $n \in \{1,2,3,4,6,10 \}$.
Indeed, the isotropy representation of $Q_{10}$ is the third one of the above table.
The following isomorphisms are well-known:
\[ Q_1 \cong \mathbb{C}\mathbb{P}^{1} \, , \,  Q_2 \cong \mathbb{C}\mathbb{P}^{1} \times \mathbb{C}\mathbb{P}^{1} \, , \,
   Q_3 \cong Sp(2)/U(2)  \, , \, Q_4 = SU(4)/S(U(2)\times U(2)) \, , \, Q_6 \cong SO(8)/U(4)\]
Then it is clear that the isotropy representations of $Q_1,Q_3,Q_4$ and $Q_6$ can be realized as normal holonomies.
The isotropy representation of $Q_2 \cong \mathbb{C}\mathbb{P}^{1} \times \mathbb{C}\mathbb{P}^{1}$ is a product and can be realized
as the normal holonomy of the smooth part of a join as explained in Theorem \ref{ConosProduct} below. 

\begin{remark}\label{referee2}
But, for all $n$, the isotropy representation of $Q_n$ is a normal holonomy of
a real submanifold of Euclidean space \cite[Theorem 1]{HO}. It is a remarkable fact that
there exist Hermitian s-representations which are normal holonomy
of real submanifolds but never arise as normal holonomy of complex
submanifolds.
\end{remark}

Therefore the classification and realization problems for complex submanifolds whose normal holonomy acts irreducibly are completely solved.

The non irreducibility of the normal holonomy group for complex submanifolds of $\mathbb{C}^n$ implies a De Rham type theorem.
It asserts roughly that if the normal holonomy group is reducible then the complex submanifold is an extrinsic product (cf. \cite{D00}).

The following result shows that for full complex submanifolds of $\mathbb{C}\mathbb{P}^n$ the non irreducibility of the normal holonomy group is related
to the concept of \emph{join} of algebraic geometry.

\begin{thm}{\label{ConosProduct}} Let $M \subset \mathbb{C}\mathbb{P}^n$ be a full complex submanifold (non nec. complete). Let $$\nu M = \nu_1 \oplus \nu_2 \oplus \cdots \oplus \nu_r$$ be the decomposition of the normal space into $Hol^*(M,\nabla^{\perp})$-irreducible subspaces. Then (locally)  there exist full complex subvarieties \[ M_1 \subset \mathbb{C}\mathbb{P}^{n_1},  \cdots , M_r \subset \mathbb{C}\mathbb{P}^{n_r} \] such that $n=n_1 + \cdots + n_r$  and $M$ is an open subset of the smooth part of  the join \[ J(M_1, M_2, \cdots , M_r) \subset \mathbb{C}\mathbb{P}^n \, .\]  Moreover, the restricted normal holonomy group  $Hol^*(M, \nabla^{\perp})$ is the product of the normal holonomies of the submanifolds $M_i \subset \mathbb{C}\mathbb{P}^{n_i}$ for $i=1,\cdots,r$.
\end{thm}

\medskip

For the proofs of our main results  we use the theory of \textsl{Hermitian Jordan triple systems}. This allows us to explicitely  describe many important aspects of the geometry of the submanifolds involved in terms of very simple and well known algebraic objects. We hope that the techniques developed here will be helpful for other problems in the theory of complex submanifolds.

\section{Preliminaries}

\subsection{Complex submanifolds and normal holonomy}
In this section we shall introduce the normal holonomy group and some properties that will be needed for the proofs of the main theorems. For a more detailed explanation see \cite{BCO}.

Let $M$ be a connected complex submanifold of the complex Euclidean space $\mathbb{C}^n$ or the complex projective space $\mathbb{C}\mathbb{P}^n$ with the Fubini-Study metric. Then at any point $p\in M$, the tangent space of the ambient manifold decomposes as the direct sum $$T_p M\oplus \nu_p M$$ where $\nu_pM=(T_p M)^{\bot}$ is the normal space of $M$ at $p$. The union $$\nu M:=\bigcup_{p\in M}\nu_pM$$ has a natural structure of a vector bundle with base manifold $M$, called the \textsl{normal bundle} of $M$.
 If $\overline{\nabla}$ is the Levi-Civita connection of the ambient manifold ($\mathbb{C}^n$ or $\mathbb{C}\mathbb{P}^n$), then the normal part of the derivative $\overline{\nabla}_X\xi$ of any section of $\nu M$ with respect to a tangent vector field $X$ to $M$, defines an affine connection on $\nu M$, called the \textsl{normal connection} of $M$ and denoted by $\nabla^{\bot}$.

Consider the holonomy groups associated to $\nabla^{\bot}$ defined in the usual way. Namely, let $\Omega_p$ be the set of loops in $M$ based at $p$ and $\Omega^0_p\subset \Omega_p$ the set of null-homotopic loops in $M$ based at $p$. If $\tau^{\bot}_\gamma$ denotes the parallel transport with respect to $\nabla^{\bot}$ along the loop $\gamma$, one has the groups
 $$Hol_p(M,\nabla^{\bot}):=\{\,\tau^{\bot}_{\gamma}\,:\,\gamma\in \Omega_p\},\ \
 Hol^*_p(M,\nabla^{\bot}):=\{\,\tau^{\bot}_{\gamma}\,:\,\gamma\in \Omega^0_p\}$$
 $Hol_p(M,\nabla^{\bot})$ and $Hol^*_p(M,\nabla^{\bot})$ are called the \textsl{normal holonomy group} and \textsl{restricted normal holonomy group} of $M$ at $p$ respectively. $Hol^*_p(M,\nabla^{\bot})$ is the connected component of the identity in $Hol_p(M,\nabla^{\bot})$. When $M$ is connected, these groups are respectively  conjugated to each other and we shall omit the base point.

The \textsl{local normal holonomy group} $Hol^{loc}_p(M,\nabla^{\bot})$ of $M$ at $p$ is the intersection of the restricted holonomy groups $Hol^*_p(U,\nabla^{\bot}_{|U})$, varying $U$ among all open neighborhoods of $p$ in $M$. There always exists a small enough open neighborhood $U'$ of $p$ such that $Hol^{loc}_p(M,\nabla^{\bot})=Hol_p^{*}(U',\nabla^{\bot}_{|U'})$.

\begin{rem}\label{analitic}
Since $M$ is an analytic submanifold and $\nabla^{\bot}$ is an analytic connection, one has $Hol^*_p(M,\nabla^{\bot})=Hol^{loc}_p(M,\nabla^{\bot})$ for each $p\in M$ (cf. \cite[Theorem 10.8, Ch. II]{KN}).
\end{rem}

\medskip

We shall denote by $\nu_0M$ the maximal parallel and flat subbundle of $\nu M$. Then $Hol^*_p(M,\nabla^{\bot})$ acts trivially on $(\nu_0 M)_p$ for each $p$.

 Set, for each $p\in M$, $(\nu_sM)_p=(\nu_0 M)_p^{\bot}\subset \nu_p M$. Then $$\nu_p M=(\nu_0 M)_p\stackrel{\bot}{\oplus} (\nu_s M)_p.$$ We will refer to $(\nu_s M)_p$ as the \textsl{semisimple part} of $\nu_p M$.

\medskip

We say that the action of the group $Hol_p^*(M,\nabla^{\bot})$ is \textsl{non transitive}, or that $M$ has \textsl{non transitive normal holonomy} if the action of $Hol^*_p(M,\nabla^{\bot})$ is non transitive on the unit sphere of $\nu_pM$.

\subsection{Normal holonomy of a union of parallel manifolds}
\label{secnormalparallel}

An important construction that we will use in the paper is the following \textsl{union of parallel submanifolds}.

 Namely, let $M$ be a submanifold of $\mathbb{R}^n$ and let $x\in M$. Let $V$ be a small ball around $x$. Then the restriction of $\nu_0 M$ to $V$ is the trivial vector bundle $V\times (\nu_0 M)_x$. So, there is a small ball $U$ centered at $0$ in $(\nu_0M)_x$ such that:
 \begin{itemize}
 \item[i)] if $\xi_x\in U$, $\xi_x$ extends to a parallel normal vector field $\xi$ to $V$;
 \item[ii)] the subset $$M_{\xi}:=\{p+\xi(p)\,:\,p\in V\}$$
is a submanifold of $\mathbb{R}^n$ called \textsl{parallel submanifold} to $M$ (cf. \cite{BCO}).
\end{itemize}
Now, define
$$N:=\bigcup_{\xi \in U}M_{\xi}\subset \mathbb{R}^n.$$
Notice that $N$ is a submanifold diffeomorphic to $V\times U$, hence the codimension of $N$ is $\textrm{dim} (\nu_s M)_x$. Indeed, the natural projection $$\pi:N\to M$$ is a submersion whose fibers $\pi^{-1}(y)$ are balls centered at $0$  in $(\nu_0 M)_y$ obtained by normal parallel transport of $U$.

The following lemma shows how to compute the normal holonomy of $N$.

\begin{lema} For every $q\in N$, $\nu_q N=(\nu_s M)_{\pi(q)}$ (as subspaces of $\mathbb{R}^n$) and the local normal holonomy group of $N$ at $q$ is the local normal holonomy group of $M$ at $\pi(q)$ acting on $(\nu_s M)_{\pi(q)}$. In particular, the subspace of fixed points of the local holonomy group of $N$ is trivial.
\label{lemaholo}
\end{lema}

\begin{proof}
The proof is similar to that of Lemma 4 in \cite{CDO} but we include it here for the readers convenience.

Let $\mathcal{H}$ and $\mathcal{V}$ be the horizontal and vertical distributions associated to the submersion $\pi$. The key observation is that $\mathcal{V}$ is included in the nullity $\mathcal{N}$ of the second fundamental form of $N$. Indeed, if $p\in \pi^{-1}(x)$, then it is standard to show that  \begin{equation} \label{tgunion}
T_p N=T_xM\oplus (\nu_0 M)_x.
\end{equation}
So the tangent space of $N$ is constant along the fibers of $\pi$, which implies that $\mathcal{V}\subset \mathcal{N}$.

 Observe that equation (\ref{tgunion}) also implies that $\nu_p N=(\nu_s M)_{x}$.

Since $\mathcal{V}\subset \mathcal{N}$, by the Ricci equation on $N$, one has that the normal curvature tensor of $N$ satisfies $$R^{N \bot}_{X,Y}\equiv 0$$ where $X$ is a horizontal vector and $Y$ is a vertical one.

Let now $\sigma$ be a loop on $N$ based at a point $q$. Since we are working locally, we can  apply the Lemma in \cite[Appendix]{Ol93} to conclude that there exist a horizontal loop $\sigma_h$ and a vertical loop $\sigma_v$ based at $q$ such that $$\tau^{\bot}_{\sigma}=\tau^{\bot}_{\sigma_v}\circ\tau^{\bot}_{\sigma_h}.$$
Since $\mathcal{V}$ is contained in the nullity of $N$, $\tau^{\bot}_{\sigma_v}$ is trivial. This shows that the local normal holonomy group of $N$ at $q$ is the local normal holonomy group of the parallel submanifold $M_{\xi}$ through $q$ acting on $(\nu_s M_{\xi})_{q}$. But from Lemma 4.4.6 in \cite[Page 120]{BCO}, this is the same as the  local normal holonomy group of $M$ acting on $(\nu_s M)_{\pi(q)}$.
\end{proof}

We will use Lemma \ref{lemaholo} in the case where $N$ and $M$ are complex submanifols and so the local normal holonomy group and the restricted normal holonomy group coincide (cf. Remark \ref{analitic}).

\subsection{Normal holonomy of orbits of $s$-representations}

We recall here some well known facts about the theory of normal holonomy groups and $s$-representations.

\begin{thm} \cite{HO}
Let $X=G/K$ be a symmetric space of noncompact type with $G=I_0(X)$, $v\in T_pX$ and let $M=K\cdot v$ be an orbit of the isotropy representation. If $M$ is full, then the normal holonomy of $M$ is equivalent to the slice representation, i.e., the action of the normal holonomy group of $M$ at $v$ is equivalent to the action of  $K_v$ on $\nu_v M$.
\label{HO} \end{thm}

Consider a group $K$ acting on $\mathbb{R}^n$ as the isotropy representation of an irreducible symmetric space. Fix $v\in \mathbb{R}^n$ and choose a normal vector $\xi\in \nu_s(K\cdot v)$. For a small $\mu\in \mathbb{R}$, consider now the orbit $K\cdot (v+\mu\xi)$. Notice that $K\cdot(v+\mu\xi)$ is a so called \textsl{holonomy tube} (see \cite[Page 124]{BCO} and cf. \cite{HOT}).

\begin{thm} \label{holiterada}
With the above notations,
\begin{enumerate}
\item $T_{v+\mu\xi}(K\cdot(v+\mu\xi))=T_v (K\cdot v)\stackrel{\bot}{\oplus} T_{\xi}(K_v\cdot \xi)$;
\item $\left(\nu_0(K\cdot(v+\mu\xi))\right)_{v+\mu\xi}=(\nu_0( K\cdot v))_{v}\stackrel{\bot}{\oplus} (\nu_0(K_v\cdot\xi))_{\xi}$, considering $K_v\cdot \xi$ as a submanifold of $(\nu_s(K\cdot v))_{v}$.
\item $(\nu_s(K\cdot(v+\mu\xi)))_{v+\mu\xi}=(\nu_s(K_v\cdot \xi))_{\xi}$, considering
$(\nu_s(K_v\cdot \xi))_{\xi}\subset (\nu_s(\,K\cdot v))_{v}$.
\item The action of the normal holonomy group of the orbit  $K\cdot(v+\mu\xi)$ on 

$(\nu_s (K\cdot (v+\mu\xi)))_{v+\mu\xi}$ coincides with the action of the iterated isotropy group $(K_v)_{\xi}$ on $(\nu_{s}(K_{v}\cdot \xi))_{\xi}$.
\end{enumerate}
\end{thm}

The proof follows combining the results in \cite{HOT}, Theorem \ref{HO} and  Theorem 5.4.12 in \cite{BCO}.

\subsection{Hermitian Jordan triple systems.}

A \textsl{Hermitian Jordan triple system} (HJTS for short) is a pair $(V,\{\ \})$, where $V$ is a complex vector space and
$$\begin{array}{cccl}
\{\ \}:& V\times V\times V&\longrightarrow& V\\
 & (x,y,z)&\longrightarrow & \{x\; y\; z\}
\end{array}$$
is a $\R$-trilinear map, which is $\mathbb{C}$-bilinear and symmetric in $x$ and $z$ and $\mathbb{C}$-antilinear in $y$ and the following \textsl{Jordan identity} holds:
$$
\{x\; y\; \{u\; v\; w\}\}=\{\{x\; y\; u\}\; v\; w\}-\{u\; \{v\; x\; y\}\; w\}+\{u\; v\;\{x\; y\; w\}\}$$
%\label{jordan}

Associated to the trilinear map $\{\ \}$ one has the operators $$D:V\times V\to End(V),\ \ Q:V\times V\to End(V)$$
defined as $$D(x,y)z=\{x\; y\; z\},\ \ Q(x,y)z=\{x\; z\; y\}$$ for $x,y,z\in V$.
We will also denote by $Q$ the quadratic form associated to $Q$, that is, $$Q(x)\;y=\frac{1}{2}Q(x,x)\; y=\frac{1}{2}\{x\; y\; x\}$$

A HJTS $V$ is called a \textsl{positive Hermitian Jordan triple system} (PHJTS for short) if for every $x\in V$ $$\mathrm{tr} D(x,x)>0.$$ In this case $(x\;|\;y)=\mathrm{tr} D(x,y)$, $x,y\in V$, defines a Hermitian inner product such that for all $x,\; y\in V$, $D(x,y)$ is a self-adjoint endomorphism of $V$.

An \textsl{ideal} of $V$ is a vector subspace $I$ of $V$ such that $$\{I\; V\; V\}\subset I,\ \ \{V\; I\; V\}\subset I$$
where $\{A\; B\; C\}$ denotes as usual the vector space spanned by all elements of the form  $\{x\;y\;z\}$ with $x\in A,\, y\in B,$ and $ z\in C$

Each ideal of $V$ is itself a HJTS and if $V=V_1\oplus V_2$ as direct sum of subspaces with $V_1$ and $V_2$ ideals, then $V$ is the direct sum of $V_1$ and $V_2$ as HJTS. This means that $$\{x\,y\,z\}=\{x_1\,y_1\,z_1\}+\{x_2\,y_2\,z_2\}$$
where $x=x_1+x_2$, $x_1\in V_1,\ x_2\in V_2$, and the same for $y$ and $z$.

We say that $V$ is \textsl{simple} if $V\neq 0$ and $V$ has no proper ideals. Any PHJTS can be decomposed (uniquely up to order) as a direct sum
\begin{equation}
V=V_1\oplus V_2\oplus\cdots\oplus V_k \label{simpledec}
\end{equation}
of simple HJTS.

\medskip

\subsection{Tripotents and rank}
 Let $V$ be a HJTS. An element $e\in V$ is called a \textsl{tripotent}
 if $$Q(e)e=\frac{1}{2}\{e\,e\,e\}=e$$
 If $e\in V$ is a tripotent, then the endomorphism $D(e,e)$ is diagonalizable with eigenvalues $0,\, 1$ and $2$. So
 \begin{equation}
 V=V_0(e)\oplus V_1(e)\oplus V_2(e)
 \label{Peirce}
 \end{equation}
where $V_{\alpha}(e)=\{z\in V\,:\, D(e,e)z=\alpha\,z\}$. The decomposition (\ref{Peirce}) is called \textsl{Peirce decomposition} of $V$ relative to the tipotent $e$.

The restriction of $Q(e)$ to $V_0(e)\oplus V_1(e)$ is $0$ and its restriction to $V_2(e)$ has $1$ and $-1$ as eigenvalues. We shall denote by $$V_2^{\pm}(e):=\{z\in V\,:\,Q(e)z=\pm z\}$$
so $V_2(e)=V_2^+(e)\oplus V_2^-(e)$ and $V_2^-(e)=iV_2^+(e)$ (cf. \cite[Prop. V.1.1, V.1.2]{Roos}).

Two tripotents $e_1$ and $e_2$ are \textsl{orthogonal} if they verify one of the following equivalent conditions (cf. \cite[Prop. V.3.1]{Roos}) $$\text{i)}\ D(e_1,e_2)=0,\ \text{ii)}\ D(e_2,e_1)=0,\ \text{iii)}\ \{e_1\,e_1\,e_2\}=0\ \text{iv)}\ \{e_1\,e_2\,e_2\}=0.$$

In this case, the endomorphisms $D(e_1,e_1)$ and $D(e_2,e_2)$ commute, the sum $e=e_1+e_2$ is a tripotent and $D(e,e)=D(e_1,e_1)+D(e_2,e_2)$. Observe that $V_0(e)$ contains all the orthogonal tripotents to $e$.

A tripotent $e$ is called \textsl{primitive} or \textsl{minimal} if it cannot be obtained  as the sum of two orthogonal tripotents. $e$ is a primitive tripotent if and only if $V_2(e)=\mathbb{C}\cdot e$ (cf. \cite[Page 503]{Roos}).

 A maximal set of primitive mutually orthogonal tripotents is called a \textsl{frame} of $V$. All frames of $V$ have the same number of tripotents and this number is called the \textsl{rank} of $V$ and is denoted by $rank(V)$.

Each element $x\in V$ admits a (unique up to order) \textsl{spectral decomposition}
\begin{equation}
x=\lambda_1e_1+\lambda_2e_2+\cdots+\lambda_pe_p
\label{spectral1}
\end{equation}
where $(e_1,\cdots,e_p)$ are mutually orthogonal (not necessarily primitive) tripotents and 

$\lambda_1>\lambda_2\>\cdots>\lambda_p>0$.

Moreover, each $x\in V$ can also be written as
\begin{equation}
x=\lambda_1e_1+\lambda_2e_2+\cdots+\lambda_re_r
\label{spectral2}
\end{equation}
where $(e_1,\cdots,e_r)$ is a frame of $V$ and $\lambda_1\geq\lambda_2\geq\cdots\geq\lambda_p\geq0$. This is also called a spectral decomposition of $x$.
The number of non-zero $\lambda_i$ in the decomposition (\ref{spectral2}) is called the \textsl{rank} of $x$ and will be denoted by $rank(x)$.

\subsection{Bounded symmetric domains and PHJTS}
\label{bsd}

Let $V$ be a finite dimensional complex vector space and let $D\subset V$ be a bounded domain. Consider the Bergman metric on $D$ and denote by $Aut(D)$ the group of biholomorphic automorphisms of $D$, which is a closed subgroup of the group of isometries of the Bergman metric of $D$. $D$ is called a \textsl{bounded symmetric domain} if for each point $z\in D$, there exists an automorphism $s_z\in Aut(D)$ such that $s_z^2=Id_{|D}$ and $z$ is an isolated fixed point of $s_z$. Then $D$ is a Hermitian symmetric space with the Bergmann metric and $s_z$ is the geodesic symmetry around $z$ (cf. \cite{Lo}).

A bounded symmetric domain $D$ is called \textsl{circled} if $0\in D$ and $z\cdot e^{it}\in D$ for every $z\in D$ and every $t\in \mathbb{R}$. It is well known that every bounded symmetric domain is isomorphic to a bounded symmetric and circled domain. Therefore we shall only consider circled domains.

A circled bounded symmetric domain $D$ is called \textsl{irreducible} if it is not isomorphic to a direct product $D'\times D''$ of lower-dimensional circled bounded symmetric domains.

Denote by $Aut^0(D)$ the connected component of the identity in $Aut(D)$ and $K=Aut^0(D)_{0}$, the isotropy group at $0$. So $D\simeq Aut^0(D)/K$. We will refer to $K$ as the \textsl{isotropy of the bounded symmetric domain $D$}.

Every bounded symmetric domain has a PHJTS associated to it such that $V=T_0 D$ (cf. \cite[$\oint$ 2]{Lo}). Conversely, given an PHJTS $V$ and a point $x\in V$, consider the spectral decomposition given by equation (\ref{spectral1}) $$x=\lambda_1e_1+\cdots+\lambda_pe_p$$
($\lambda_1>\cdots>\lambda_p>0$) and define $|x|=\lambda_1$. Then the map
$x\mapsto |x|$
is a norm on $V$ called \textsl{spectral norm} of the PHJTS $V$.
The open unit ball $D$ of this norm is a bounded symmetric domain such that $V$ is the PHJTS associated to it (cf. \cite[$\oint 4$]{Lo}, \cite[Sec. VI.4]{Roos}).

Given a circled bounded symmetric domain $D$, we call \textsl{rank} of $D$ and denote it by $rank(D)$, the rank of the PHJTS $V$ associated to it (actually, this idea of rank coincides with the geometric definition of rank of $D$ as a symmetric space, see Remark \ref{equalrank}).
\medskip

Now, given a PHJTS $V$, an \textsl{automorphism} $f:V\to V$ of the PHJTS $V$ is a complex linear isomorphism preserving the triple product, i.e.,
$$f\{u\,v\,w\}=\{fu\,fv\,fw\}.$$
We will denote by $Aut(V)$ the group of automorphisms of $V$. It is a compact Lie group whose Lie algebra is the algebra $Der(V)$ of \textsl{derivations} of $V$, i.e., the space of complex linear maps $T:V\to V$ such that
$$T\{u\,v\,w\}=\{Tu\, v\,w\}+\{u\, Tv\, w\}+\{u\,v\,Tw\}.$$
It is easy to see that $iD(x,x)$ is a derivation of $V$ and the subspace of derivations generated by $iD(x,x)$, $x\in V$, forms a Lie subalgebra $Int(V)$ of $Der(V)$, called the algebra of \textsl{inner derivations}. Then for every $x,y\in V$, $$D(x,y)-D(y,x)\in Int(V)$$ (see \cite[Page 518]{Roos}). Morover, $Int(V)$ is generated by these derivations.

\medskip

We summarize in the following theorem some useful results that relate some geometrical aspects of a bounded symmetric domain $D$ with some properties of the associated PHJTS $V$.

\begin{thm} Let $D$ be a bounded symmetric domain and let $V$ be the PHJTS associated to it. Let $K$ be the isotropy of $D$.  Then:
\label{props}
\begin{enumerate}
\item \cite[Page 4.11]{Lo} $D$ is irreducible if and only if $V$ is simple.
\item \cite[Theorem 2.10]{Lo} If $R$ is the curvature tensor of $D$, then for every $x,y\in V$
$$R(x,y)=D(y,x)-D(x,y).$$
\item \cite[Cor. 4.9]{Lo} $K=Aut^0(V)$, where $Aut^0(V)$ is the connected component of the identity in $Aut(V)$.
\end{enumerate}

\end{thm}

Finally, consider the complexification $\Gamma$ of the group $K$. Then the orbits of $\Gamma$ in $V$ are precisely the sets of elements of the same rank (cf. \cite[Page 253]{KA}.

\subsection{Some basic properties}

In this section we present some basic properties of bounded symmetric domains and PHJTS that will be usefull later.

\begin{lema} \label{lemmarankopen}
Let $D \subset \mathbb{C}^{n}$ be a circled bounded symmetric domain and let $M$ be a submanifold of $V=T_0 D = \mathbb{C}^{n}$. The set of vectors of maximal rank of $M$ is an open subset of $M$.
\end{lema}
\begin{proof}
For $x\in V$, let $x=\lambda_1e_1+\lambda_2e_2+\cdots+\lambda_re_r$ be the spectral decomposition of $x$ with respect to a system $(e_1,\cdots,e_r)$ of mutually orthogonal primitive tripotents (i.e., a frame), with $\lambda_1\geq \lambda_2\geq \cdots\lambda_r \geq 0$. So $rank(x)=k$ if and only if $\lambda_{k+1}=\cdots=\lambda_r=0$.

There exist polynomials $m_1,\cdots,m_r$ on $V\times \overline{V}$ homogeneous of respective bidegrees $(1,1),\cdots,(r,r)$ such that the generic polynomial $$m(T,x,y)=T^{r}-m_1(x,y)T^{r-1}+\cdots+(-1)^{r}m_r(x,y)$$
satisfies \cite[page 515]{Roos} $$m(T,x,x)=\prod_{i=1}^{r}(T - \lambda_i^{2})$$
So $rank(x)=k$ if and only if $m_1(x,x)\neq 0,\ \cdots,\ m_k(x,x) \neq 0$, and $m_{k+1}(x,x)=m_{k+2}(x,x)=\cdots =m_r(x,x)=0$.

If $rank(x)=k$, then there exist a neighbourhood $U$ of $x$ such that $m_i(y,y)\neq 0$ for $i=1,\cdots k$, for all $y\in U$. We conclude that $rank(y)\geq k$ for all $y\in U$. Therefore if $x$ is of maximal rank $k$ in $M$, then $rank(y)=k$ for all $y\in U\cap M$.
\end{proof}
\begin{lema}
Let $V$ be a PHJTS. Then all derivations of $V$ are inner derivations.
\label{lemmainner}
\end{lema}
\begin{proof}
Let $D$ be the circled bounded symmetric domain associated to $V$ and let $K$ be its isotropy group. Let $Der(V)$ be the set of derivations of $V$. Then $Der(V)$ is the Lie algebra of $Aut^0(V)=K$  (cf. Theorem \ref{props}).

On the other hand, if $R$ is the curvature tensor of $D$, then the Lie algebra of $K$ is generated by the operators of the form $R_{x,y}$, $x,y\in V=T_0 D$.

But again by Theorem \ref{props} we have $$R_{x,y}=D(y,x)-D(x,y)$$
which generate the Lie algebra of inner derivations.
\end{proof}
\begin{lema}
Let $D$ be a circled bounded symmetric domain and let $e$ be a tripotent of the PHJTS $V=T_0 D$ associated to it. Let $K$ be the isotropy group of $D$ and $K_e$ the isotropy group at $e$ of $K$ acting on $V$. Then $K_e$ is the connected component of the identity of the group of automorphisms of the PHJTS $V_0(e)$, i.e., $K_e=Aut^0(V_0(e))$.
\label{KV0}
\end{lema}
\begin{proof}

First of all notice that $V_0(e)$ is a PHJTS, and that the spectral decomposition of any element $x\in V_0(e)$ in $V_0(e)$ coincides with its spectral decomposition as an element of $V$ (cf. \cite[Prop. VI.2.4]{Roos}).

Observe now that  $V_0(e)$ is invariant under the action of $K_e$. Indeed, $x\in V_0(e)$ if and only if $\{e\;e\;x\}=0$ and since by Theorem \ref{props}, $K=Aut^0(V)$ one gets $$0=\{k\cdot e\; k\cdot e\;k\cdot x\}=\{e\;e\; k\cdot x\}\ . $$
We conclude that  $$K_e\subset Aut^0(V(e))\ .$$
 We shall see that the Lie algebra $Der(V_0(e))$ of $Aut^0(V(e))$ is contained in the Lie algebra $\mathfrak{k}_e$ of $K_e$.
By Lemma \ref{lemmainner}, $Der(V_0(e))$ is generated by elements of the form $iD(x,x)_{|V_0(e)}$ for $x\in V_0(e)$.
On the other hand, $\mathfrak{k}_e$ is generated by the derivations $A$ of the PHJTS $V$ such that $T(e)=0$. Therefore we only need to prove that $iD(x,x)e=0$ for every $x\in V_0(e)$.

Fix $x\in V_0(e)$ and consider its spectral decomposition. There exist primitive orthogonal tripotents $e_1,\cdots,e_k$ of $V$, all perpendicular to $e$ (since they must be on $V_0(e)$) such that $x=\lambda_1e_1+\cdots+\lambda_ke_k$. So $$\{x\; x\; e\}=\sum_{i=1}^{k}\lambda_i\lambda_j\{e_i\;e_j\;e\}=0\ .$$
Then $iD(x,x)e=0$ as we wanted to show.
\end{proof}
\begin{lema}
Let $V$ be a simple PHJTS and let $e$ be a minimal tripotent of $V$. Then $V_0(e)$ is also a simple PHJTS.
\label{V0}
\end{lema}
\begin{proof}
Assume $V_0(e)$ is not simple. Then it splits as the sum of at least two ideals $A$ and $B$.
Let $D_0$, $D_A$ and $D_B$ be the bounded symmetric domains associated to $V_0(e)$, $A$ and $B$ respectively. Then $D_0=D_A\times D_B$ (see \cite[Page 4.11]{Lo}) and since $K_e$ is the holonomy group of $D_0$, it splits accordingly, i.e., $K_e=K_A\times K_B$, such that $K_A$ is the holonomy group of $D_A$ and acts trivially on $B$, and the same with $K_B$.

So $A$ and $B$ are $K_e$ invariant.

Consider a frame $\{e,e_2,\cdots, e_r\}$ of $V$. So $\{e_2,\cdots,e_r\}$ is a frame of $V_0(e)$ and it can be chosen such that $\{e_2,\cdots,e_k\}$ is a frame of $A$ and $\{e_{k+1},\cdots, e_r\}$ is a frame of $B$. But since $V$ is simple,  from \cite[Theorem VI.3.5]{Roos}, there exists an element $k\in K_e$ that interchanges $e_1$ and $e_{k+1}$, which contradicts the fact that $A$ and $B$ are $K_e$ invariant. So $V_0(e)$ is simple.
\end{proof}

\subsection{Mok's characteristic varieties}
\label{Mok}

Throughout this paper we say that $M$ is a complex manifold or submanifold when it has an holomorphic differential structure in the sense of differential geometry.

By an algebraic variety $\tilde{X}\subset \mathbb{C}^{n+1}$ we intend the zero locus of a collection of polynomials in $\mathbb{C}[z_1,\cdots,z_{n+1}]$. By a projective variety $X\subset \mathbb{C}\mathbb{P}^n$ we mean the zero locus of a collection of homogeneous polynomials. That is to say, if $S$ is a subset of $\mathbb{C}[z_1,\cdots, z_{n+1}]$ consisting of homogeneous polynomials, then using homogeneous coordinates in $\mathbb{C}\mathbb{P}^n$ we have
$$X=\{x=[z_1,\cdots,z_{n+1}]:f(z_1,\cdots,z_{n+1})=0,\ \forall\, f\in S\}$$
If $\pi:\mathbb{C}^{n+1}-\{0\}\to \mathbb{C}\mathbb{P}^n$ is the usual projection and $X$ is a projective variety, then it is immediate that $$\mathcal{C}X:=\pi^{-1}(X)=\{\lambda\cdot x\,:\, \pi(x)\in X,\ \lambda\in \mathbb{C}-\{0\}\}$$ is an algebraic variety of $\mathbb{C}^{n+1}$ called the \textsl{cone} over $X$.

\medskip

The set of smooth points of an algebraic or projective variety will be denoted by $X_{sm}$ and will be called \textsl{the smooth part} of $X$. Then $X_{sm}$ is a complex submanifold which is open and dense in $X$ with the Zariski topology (see \cite{Ha92}).

\medskip

Consider now a PHJTS with $V=\mathbb{C}^{n+1}$ associated to a circled bounded symmetric domain $D\subset \mathbb{C}^{n+1}$. We call \textsl{$j^{th}$-Mok's characteristic variety} the set $$S^j(D):=\{\pi(x):1\leq rank(x)\leq j\} \subset \mathbb{C}\mathbb{P}^n$$
which is actually a projective variety (cf. \cite[Page 252]{MokLibro}).

Moreover, each $S^j(D)_{sm}$ is an orbit of the complexification $\Gamma$ of the isotropy group $K$ of $D$ and $S^1(D)$ is the only smooth variety among all Mok's characteristic varieties (cf. \cite[Lemma 2.3, Page 572]{CD14}).

\subsection{Geometry and Algebra of the first Mok's characteristic variety}
\label{firstmok}

Let $D\subset \mathbb{C}^n$ be an irreducible circled bounded symmetric domain with $rank(D)\geq 2$, and let $K$ be its isotropy.

  Notice that the first Mok's characteristic variety $S^1(D)$ is the unique complex orbit of $K$ in the complex projective space $\mathbb{C}\mathbb{P}^n$ whose normal holonomy was computed in \cite{CD}.

 Locally, the cone $\mathcal{C}S^1(D)$ can be described as the union of parallel submanifolds
                                  \[ E_1:=\bigcup_{\xi \in U} (K\cdot e_1)_{\xi} \]
where $U$ is a small open neighbourhood of $0$ in $(\nu_0 (K\cdot e_1))_{e_1}$ and $e_1$ is a primitive tripotent.
Indeed, in terms of the Peirce decomposition (\ref{Peirce}) with respect to the tripotent $e_1$,
\begin{equation}\label{tangente}
T_{e_1}(K\cdot e_1) = V_2^{-}(e_1) \oplus V_1(e_1)
\end{equation} (cf. \cite[Theorem 5.6]{Lo}). Since $e_1$ is primitive $V_2(e_1)=\mathbb{C}\cdot e_1$. On the other hand, $(\nu_0(K\cdot e_1))_{e_1}$ has dimension 1 (cf. the Proof of Theorem 5 in \cite{CDO}) and so
\begin{equation}\label{normal} (\nu_0(K\cdot e_1))_{e_1}=\mathbb{R}\cdot e_1=V_2^{+}(e_1).
\end{equation}
 Therefore, $T_{e_1}E_1=V_1(e_1)\oplus V_2(e_1)$. This shows that $E_1$ is a complex submanifold (locally) invariant by the complexification $\Gamma$ of the group $K$. Hence all the vectors in $E_1$ have rank $1$, and so $E_1$ is an open subset of $\mathcal{C}S^{1}(D)$.

\medskip

With respect to the normal holonomy of $\mathcal{C}S^1(D)$, which is the same as the normal holonomy of $E_1$ by Remark \ref{analitic}, observe that by Lemma \ref{lemaholo}$$\nu_{e_1}(\mathcal{C}S^1(D))=(\nu_s(\mathcal{C}S^1(D)))_{e_1}=(\nu_s(K\cdot e_1))_{e_1}=V_0(e_1).$$
By lemma \ref{V0},  $V^1:=V_0(e_1)$ is itself a simple PHJTS of rank $$rank(V_0(e_1))=rank(V)-1.$$ Let $D_1\subset V^1$ be its associated irreducible circled bounded symmetric domain and let $K_1$ be its isotropy group. Then by Lemma \ref{KV0}, $K_1=K_{e_1}$ and so, by Theorem \ref{HO} and Lemma \ref{lemaholo} the normal holonomy group of  $\mathcal{C}S^1(D)$ at $e_1$ is $K_1$.

Then if $rank(D)\geq 3$, the normal holonomy of $\mathcal{C}S^1(D)$ is non transitive.\\

Taking into account that the normal holonomy of the cone $\mathcal{C}S^1(D)$ is the same as the normal holonomy of $S^1(D)$ the above argument
gives a different proof of the results in \cite{CD} using the language of Jordan Triple Systems.\\

It is well-known that $S^1(D)$ is a submanifold with parallel second fundamental form (see \cite{NT76}).
As an application of the language of the Jordan Triple Systems we give a very simple proof of this fact.

\begin{prop} $S^1(D)$ is extrinsically symmetric hence has parallel second fundamental form.
\end{prop}
\begin{proof} The proof consists in an explicit construction of the extrinsic symmetry $\sigma$.
According to equations (\ref{tangente}) and  (\ref{normal}):
\[ T_{\pi(e_1)}S^1(D) = V_1(e_1) \, \, \mbox{ and }\, \, \nu_{\pi(e_1)}(S^1(D))=V_0(e_1) \]
So the extrinsic symmetry $\sigma$ must satisfy
\[ \sigma|_{V_1(e_1)} = -\mathrm{Id}  \, \, \mbox{ and }\, \, \sigma|_{V_0(e_1)} = \mathrm{Id} \, \, .\]

Then we define $\sigma \in \mathrm{End}(V)$ such that the above condition holds and $\sigma|_{V_2(e_1)} = \mathrm{Id}$.

Now it is a straightforward computation to check that for all $x,y,z \in V$ the following holds:
\[ \sigma \{ x \, y \,z \} = \{ \sigma x \, \sigma y \, \sigma z \} \, .\]
Thus, $\sigma \in Aut(V)$. Then $\sigma$ induces an isometry of $\mathbb{P}(V)$ which preserves
$S^1(D)$.
\end{proof}

\section{Proof of the main theorems}

\subsection{Proof of Theorem \ref{Euclideo}}
\begin{proof}
Let $M \subset \mathbb{C}^n$ be a full and irreducible, complex submanifold such that the action of $Hol^*(M, \nabla^{\perp})$  is non transitive on the unit sphere of the normal space. Observe that since $M$ is irreducible, then $Hol^*(M, \nabla^{\perp})$ acts irreducibly on the normal space (cf. \cite{D00}).

  According to \cite[Theorem 4]{CDO}, there exists an irreducible bounded symmetric domain $D \subset \mathbb{C}^{n}$ (realized as a circled domain) such that, locally around a generic point $q$, $M$ may be described as the union of orbits of the isotropy group $K$ of $D$. More preciselly, in a neighbourhood of $q$,
  \begin{equation}
  M=\bigcup_{v\in\nu_0(K\cdot q)}(K\cdot q)_{v}\ .
  \label{union}
  \end{equation}

  Consider the set $W$ of points of maximal rank of $M$ and apply the decomposition (\ref{union}) in a neighbourhood $U$ of a point $q\in W$. Since by Lemma \ref{lemmarankopen} $W$ is open, $U$ can be chosen so that all its points have the same rank.

Let $\Gamma$ be the complexification of $K$. Since $M$ is a complex manifold,  $M$ is (locally) invariant under the action of $\Gamma$. But, from \cite[pg. 253]{KA}, the orbits of $\Gamma$ are the set of all vectors of the same rank. This implies that $U$ is an open subset of the smooth part of the Mok's characteristic cone $\mathcal{C}S^j(D)$ for some $1 \leq j < rank(D) - 1$.

Now, $\mathcal{C}S^j(D)$  is an algebraic variety. Since $M$ is analytic we get that the whole $M$ is an open subset of  $\mathcal{C}S^j(D)_{sm}$.

\vspace{0.5cm}

We will prove the converse by induction on $j$. In order to better illustrate the procedure we shall do first the proof for the case $j=2$, based on the construction made for the first Mok's characteristic cone in section \ref{firstmok}.

 Let $D\subset \mathbb{C}^n$ be an irreducible circled bounded symmetric domain and let $K$ be its isotropy group.

As we have shown in Section \ref{firstmok}, the normal holonomy of $\mathcal{C}S^1(D)$ is non transitive if $rank(D)\geq 3$.

 Consider now $\mathcal{C}S^2(D)$, the cone over the second Mok's characteristic variety and assume that $rank(D)\geq 4$.  We are going to construct an open subset $E_2$ of  $\mathcal{C}S^2(D)_{sm}$ in the same way as we did when we constructed $E_1\subset \mathcal{C}S^1(D)$.

We will keep the notations of Section \ref{firstmok}.  Let now $e_2$ be a primitive tripotent of $V^1$. Then for a small number $\mu_2$ we can construct the (local) submanifold
\[ E_2 := \bigcup_{\xi \in  U_2} (K\cdot(e_1 + \mu_2 e_2))_{\xi} \, \, \]
where $U_2$ is a small open neighborhood of $0$ in $(\nu_0(K\cdot(e_1 + \mu_2e_2)))_{e_1 + \mu_2e_2}$.

From Theorem \ref{holiterada},  the tangent space of $K\cdot(e_1 + \mu_2 e_2)$ is $$T_{e_1+\mu_2e_2}(K\cdot (e_1+\mu_2e_2))=T_{e_1}(K\cdot e_1)\oplus T_{e_2}(K_1\cdot e_2)$$
 and $$(\nu_0 (K\cdot (e_1+\mu_2e_2)))_{e_1+\mu_2e_2}=\mathbb{R}\cdot e_1\oplus\mathbb{R}\cdot e_2.$$
Observe that both subspaces $T_{e_1}(K\cdot e_1) \oplus \mathbb{R}\cdot e_1 $ and $T_{e_2}(K_1\cdot e_2) \oplus \mathbb{R}\cdot e_2$ are complex subspaces of $V$ as we explained in subsection \ref{firstmok}. Indeed, they are the tangent spaces to the respective first Mok's characteristic cone.

So $T_{e_1+\mu_2e_2}E_2 = T_{e_1}(K\cdot e_1) \oplus \mathbb{R}\cdot e_1 \oplus T_{e_2}(K\cdot e_2) \oplus \mathbb{R}\cdot e_2 $ is complex. Hence $E_2$ is a complex submanifold (locally) invariant by the complexification $\Gamma$ of the group $K$. So $E_2$ is open in $\mathcal{C}S^2(D)_{sm}$ since $rank(e_1+\mu_2e_2)=2$.

\vspace{0.3cm}

We determine now the normal holonomy of $E_2$.
Again by Lemma \ref{holiterada} $$(\nu_s(K\cdot(e_1+\mu_2e_2)))_{e_1+\mu_2e_2}=(\nu_s (K_1\cdot e_2))_{e_2}\subset V^1$$ and the action of the normal holonomy of $K\cdot(e_1+\mu_2e_2)$ on the semisimple part of its normal space coincides with the action of $(K_1)_{e_2}$ on $(\nu_s (K_1\cdot e_2))_{e_2}$.

 With the same argument made for $E_1$ in section \ref{firstmok}, we conclude that $$(\nu_s (K_1\cdot e_2))_{e_2}=V^1_0(e_2),$$ where $V^1_0(e_2)$ is the $0$-space associated to the Peirce decomposition of $V^1$ relative to the tripotent $e_2$ and $K_2:=(K_1)_{e_2}$ is the isotropy group of the bounded symmetric domain associated to the simple PHJTS $$V^2:=V^1_0(e_2)=V_0(e_1)\cap V_0(e_2)=V_0(e_1+e_2).$$

Hence by Lemma \ref{lemaholo} and Remark \ref{analitic}, the restricted normal holonomy group of $\mathcal{C}S^2(D)_{sm}$ is the isotropy group of the bounded symmetric domain associated to $V^{2}$. Observe that the rank of $V^{2}$ is different from 1, since otherwise we would have $rank(D)\leq 3$. It then follows that the normal holonomy of $\mathcal{C}S^2(D)_{sm}$ is non transitive.

\medskip

We now prove the general case by induction, repeating the arguments used for $\mathcal{C}S^2(D)$. Our inductive statement is the following:

\textsl{Assume $rank(D)=r$ and fix $1\leq j<r-1$. Then there exist orthogonal primitive tripotents $e_1,\cdots,e_{j}$ and (small) real numbers $\mu_2,\cdots,\mu_{j}$ such that
 if $$y_{j}=e_1+\mu_2e_2+\cdots+\mu_{j} e_{j},$$ then
 \begin{enumerate}
 \item[i)] $\displaystyle (\nu_0 (K\cdot y_{j}))_{y_{j}}=\mathbb{R}\cdot e_1\oplus\cdots\oplus\mathbb{R}\cdot e_{j}$;
     \vspace{0.2cm}
 \item[ii)]
$\displaystyle E_{j}:=\bigcup_{\xi\in U_{j}} (K\cdot y_{j})_{\xi}$
 is an open submanifold of $\mathcal{C}S^{j}(D)_{sm}$, where $U_{j}$ is a small open neighborhood of $0$ in $(\nu_0 (K\cdot y_{j}))_{y_{j}}$;
 \vspace{0.2cm}
 \item[iii)] $\nu _{y_{j}}E_{j}=V^{j}:=V_0(e_1+\cdots+e_{j})$ is a simple PHJTS and the restricted normal holonomy group of $E_j$ is the isotropy group $K_j$ of the circled bounded symmetric domain $D_j$ of rank $r-j$ associated to $V^j$.
 \end{enumerate}}
 \medskip

Then, since $r-j\geq 2$, the normal holonomy of $\mathcal{C}S^j(D)_{sm}$, which by Lemma \ref{lemaholo} is the same as that of $E_j$, is non transitive.

\medskip

Observe that the case $j=1$ was proved in section \ref{firstmok}. Fix then $2\leq j<r-1$ and assume that the above statement is true for $j-1$.
 Choose a primitive tripotent $e_{j}\in V^{j-1}$, a small real number $\mu_{j}$ and set $y_{j}:=y_{j-1}+\mu_{j}e_{j}$. Consider now the submanifold
$$E_j:=\bigcup_{\xi\in U_{j}} (K\cdot y_{j})_{\xi}$$
where $U_{j}$ is a small open neighborhood of $0$ in $(\nu_0 (K\cdot y_{j}))_{y_{j}}$ so that $E_j$ is a submanifold of $\mathbb{C}^n$.

 Then by Lemma \ref{holiterada} and item i) of the induction hypothesis ,
 $$(\nu_0 (K\cdot y_{j}))_{y_{j}}=\mathbb{R}\cdot e_1\oplus\cdots\oplus\mathbb{R}\cdot e_{j-1}\oplus \mathbb{R}\cdot e_{j}$$ and
$$T_{y_{j}}(K\cdot y_{j})=T_{y_{j-1}}(K\cdot y_{j-1})\oplus T_{y_{j}}(K_{j-1}\cdot e_{j}).$$
On the other hand, $$ T_{y_j}E_j=T_{y_{j}}(K\cdot y_{j})\oplus (\nu_0 (K\cdot y_{j}))_{y_{j}}.$$
Observe that $T_{y_{j-1}}(K\cdot y_{j-1})\oplus(\mathbb{R}\cdot e_1\oplus\cdots \oplus \mathbb{R}\cdot e_{j-1})=T_{y_{j-1}}E_{j-1}$ which is a complex subspace by item iii) of the induction hypothesis.
 Moreover, $T_{y_{j}}(K_{j-1}\cdot e_{j})\oplus \mathbb{R}\cdot e_j$ is also complex, since it is the tangent space at $y_j$ of the first Mok's characteristic cone over the domain $D_{j-1}$ associated to the PHJTS $V^{j-1}$.

We conclude that $T_{y_j}E_j$ is complex. Hence $E_j$ is a complex submanifold (locally) invariant by the complexification $\Gamma$ of the group $K$. So $E_j$ is open in $\mathcal{C}S^j(D)_{sm}$ since $rank(y_j)=j$.

\vspace{0.3cm}

  To compute the normal holonomy of $E_j$, recall that by Lemma \ref{lemaholo} and Lemma \ref{holiterada} $$\nu_{y_j}E_j=(\nu_s(K\cdot y_j))_{y_j}=(\nu_s (K_{j-1}\cdot e_j))_{e_j}\subset V^{j-1}.$$
  By Lemma \ref{lemaholo} the action of the normal holonomy group of $E_j$ on $\nu_{y_j}E_j$ coincides with the action of the normal holonomy of $K\cdot y_j$ on $(\nu_s(K\cdot y_{j}))_{y_j}$. This last action coincides, by Lemma \ref{holiterada}, with the action of the iterated isotropy group  $(K_{j-1})_{e_j}$ on $(\nu_s (K_{j-1}\cdot e_j))_{e_j}$.

  With the same argument as before, we conclude that $(\nu_s (K_{j-1}\cdot e_j))_{e_j}=V^{j-1}_0(e_j)$, where $V^{j-1}_0(e_j)$ is the $0$-space associated to the Peirce decomposition of $V^{j-1}$ relative to the tripotent $e_{j}$ and $K_j:=(K_{j-1})_{e_j}$ is the isotropy group of the bounded symmetric domain associated to the simple PHJTS $$V^j:=V^{j-1}_0(e_j)=V_0(e_1+\cdots +e_j).$$

Hence the normal holonomy group of $E_j$ is the isotropy group $K_j$ of the bounded symmetric domain associated to $V^{j}$ as we wanted to show.
\end{proof}

\begin{rem} \label{equalrank}
Observe that the above construction gives a conceptual simple proof of the fact that the geometric rank of a bounded symmetric domain $D$, defined as the codimension of a principal orbit of $K$, coincides with the rank of the PHJTS $V=T_0D$ associated to it.

Indeed, let $rank(V)=r$ and let $\{e_1,\cdots,e_r\}$ be a frame of $V$, where $e_1,\cdots,e_r$ are chosen as in the previous proof. Set $y=e_1+\mu_2e_2+\cdots+\mu_r e_r$ and consider the orbit $K\cdot y$. Then $K\cdot y$ has flat normal bundle, since the orbit $K\cdot (e_1+\mu_2e_2+\cdots+\mu_{r-1} e_{r-1})$ has transitive normal holonomy on the semisimple part of its normal space. Therefore, $K\cdot y$ is principal (cf. \cite[Theorem 5.4.1]{BCO}).   Moreover, $$\nu_y (K\cdot y)= \nu_0(K\cdot (e_1+\cdots+\mu_re_r))=\mathbb{R}\cdot e_1\oplus\cdots\oplus \mathbb{R}\cdot e_r$$ and therefore the geometric rank of $D$ is $r$.
\end{rem}

\subsection{Proof of Theorems \ref{Projectivo} and \ref{ConosProduct}}
$ $

Here is the proof of Theorem \ref{Projectivo}.

\begin{proof} Let $M\subset \mathbb{C}\mathbb{P}^n$ be a full complex submanifold with irreducible and non transitive normal holonomy.  Let $\mathcal{C}M \subset \mathbb{C}^{n+1}$ be the cone over $M$. Denote by $\pi:\mathbb{C}^{n+1}\to \mathbb{C}\mathbb{P}^n$ the usual projection.

Let $p \in \mathcal{C}M$. According to \cite[Remark 5, page 211]{CDO} the action of the normal holonomy group of the cone $\mathcal{C}M$ at $p$ is the same as the action of the normal holonomy group of $M$ at $\pi(p)$. By Theorem \ref{Euclideo}, $\mathcal{C}M$ is an open subset of a cone $\mathcal{C}S^j(D)_{sm}$ over a Mok's characteristic variety, for some irreducible circled bounded symmetric domain $D$ and some $1\leq j<rank(D)-1$. Then $M$ is an open subset of the smooth part of the Mok's characteristic variety $S^j(D)$.

 \medskip

 Reciprocally, if $M$ is an open subset of the smooth part of the Mok's characteristic variety $S^j(D)$ for $1 \leq j < rank(D) - 1$ the normal holonomy group acts irreducibly but not transitively on the unit sphere as it follows again from Theorem \ref{Euclideo} and  \cite[Remark 5, page 211]{CDO}.
\end{proof}

We give now  the proof of Theorem \ref{ConosProduct}.

\begin{proof} Let $M\subset \mathbb{C}\mathbb{P}^n$ be a full complex submanifold and let $\mathcal{C}M \subset \mathbb{C}^{n+1}$ be the cone over $M$.
Let $p \in \mathcal{C}M$ and $\pi(p)$ its projection to $M$. According to \cite[Remark 5, page 211]{CDO} the action of the normal holonomy group of the cone $\mathcal{C}M$ at $p$ is the same as the action of the normal holonomy group of $M$ at $\pi(p)$.

So the normal space of $\mathcal{C}M$ at $p$ splits as $$\nu(\mathcal{C}M) = \nu_1 \oplus \nu_2 \oplus \cdots \oplus \nu_r$$ where each $\nu_j$ , $j=1,\cdots,r$ is invariant by the normal holonomy group. Then by \cite{D00} the cone  $\mathcal{C}M$ split (locally around $p$) as an extrinsic product of $r$ complex submanifolds $\mathcal{C}M_j \subset \mathbb{C}^{n_j}$,$ j=1,\cdots,r$. The meaning of such splitting is that the submanifolds $\mathbb{C}^{n_j}$, $j=1,\cdots,r$ are affine subspaces of $\mathbb{C}^{n+1}$ and locally around $p \in \mathcal{C}M$, we have
\[ \mathcal{C}M =\mathcal{C} M_1 \times \cdots \times\mathcal{C} M_r \subset \mathbb{C}^{n_1} \times \cdots \times \mathbb{C}^{n_r} =  \mathbb{C}^{n+1}  \, . \]
Since $\mathcal{C}M $ is a cone it follows that each $\mathcal{C}M_j$, $ j=1,\cdots,r$ is an open subset of a cone which we also denote by $\mathcal{C}M_j$. This shows that $\pi(p) \in M$ has a neighborhood which is open in the join  $J(M_1, M_2, \cdots , M_r)$ defined as the union of the lines joining  the projective submanifolds $M_1, M_2, \cdots , M_r$ associated to the cones $\mathcal{C}M_j$, \cite[page 70]{Ha92}.
\end{proof}

\begin{rem} Notice that even if the Riemannian metric on the cone $\mathcal{C}M$ induced by the flat metric of $\mathbb{C}^{n+1}$ is locally a product, the Riemannian metric on the join $J(M_1, M_2, \cdots , M_r)$ induced by the Fubini-Study metric of $\mathbb{C}\mathbb{P}^n$ can be locally irreducible.
\end{rem}

\begin{cor} Let $X \subset \mathbb{C}\mathbb{P}^n$ be a projective variety. Then $X$ is a join if and only if $X$ is projectively equivalent to a variety $X'$ whose normal holonomy group, defined in the smooth Zariski open subset $X'_{sm}$, does not act irreducibly on the normal space.
\end{cor}

\vspace{1cm}

{\bf Acknowledgments.} We would like to thank referee\#2 for the deep report, the useful comments (in particular Remark \ref{referee2})
and the list of typos.

\medskip

\noindent
 A.J. Di Scala is member of GNSAGA of INdAM.

\noindent
F. Vittone was partially supported by ERASMUS MUNDUS ACTION 2 programme, through the EUROTANGO II Research Fellowship, PICT 2010-1716 Foncyt and CONICET.

\noindent
The second author would like to thanks Politecnico di Torino for the hospitality during his research stay.

\vspace{1cm}

\begin{center}
\begin{tabular}{lcl}
 Antonio J. Di Scala,&$\qquad$ & Francisco Vittone,\\
 \footnotesize Dipartimento di Scienze Matematiche&$\qquad$ &\footnotesize Depto. de Matem\'atica, ECEN, FCEIA,\\
 \footnotesize Politecnico di Torino,&$\qquad$ &\footnotesize Universidad Nac. de Rosario - CONICET\\
\footnotesize Corso Duca degli Abruzzi, 24&$\qquad$ &\footnotesize Av. Pellegrini 250\\
\footnotesize 10129 Torino, Italy &$\qquad$ &\footnotesize 2000, Rosario, Argentina \\
\footnotesize\href{mailto:antonio.discala@polito.it}{antonio.discala@polito.it}&$\qquad$ & \footnotesize\href{mailto:vittone@fceia.unr.edu.ar}{vittone@fceia.unr.edu.ar}\\
\footnotesize\url{http://calvino.polito.it/~adiscala/} &$\qquad$ &\footnotesize\url{www.fceia.unr.edu.ar/~vittone}
\end{tabular}
\end{center}

\end{document}